\documentclass[11pt, onecolumn]{article}
\usepackage[top=1in, bottom=1in, left=1.25in, right=1.25in]{geometry}
\usepackage{amsfonts}
\usepackage{amsmath,amssymb}
\usepackage{amsthm}
\usepackage{graphicx}
\usepackage{color, soul}
\usepackage{bm}
\usepackage{booktabs}
\usepackage{flushend}

\theoremstyle{definition}

\theoremstyle{proposition}

\theoremstyle{lemma}
\newtheorem{lemma}{Lemma}

\theoremstyle{corollary}
\newtheorem{corollary}{Corollary}

\theoremstyle{example}
\newtheorem{example}{Example}

\theoremstyle{theorem}

\theoremstyle{remark}
\newtheorem{remark}{Remark}

\begin{document}

\title{Approximation of Gram-Schmidt Orthogonalization \\by Data Matrix}
    
\author{Gen~Li~and~Yuantao~Gu 
\thanks{The authors are with the Department of Electronic Engineering, Tsinghua University, Beijing 100084, China. The corresponding author of this work is Yuantao Gu (e-mail: gyt@tsinghua.edu.cn).}}

\date{Submitted December 31, 2016}
\maketitle

\begin{abstract}
	For a matrix ${\bf A}$ with linearly independent columns,
	this work studies to use its normalization $\bar{\bf A}$ and ${\bf A}$ itself to approximate
	its orthonormalization $\bf V$.
	We theoretically analyze the order of the approximation errors
	as $\bf A$ and $\bar{\bf A}$ approach ${\bf V}$, respectively.
	Our conclusion is able to explain the fact that a high dimensional Gaussian matrix can 
	well approximate the corresponding truncated Haar matrix.
	For applications, this work can serve as a foundation of a wide variety of problems in signal processing
	such as compressed subspace clustering.
	
	
\end{abstract}

\section{Introduction}

Suppose that $\mathcal S$ is a $d$-dimensional subspace in $\mathbb{R}^n$.
The columns of ${\bf A}=\left[{\bf a}_1,{\bf a}_2,\dots,{\bf a}_d\right]\in\mathbb{R}^{n\times d}$
constitute a basis of $\mathcal S$.
We can normalize the columns of ${\bf A}$ and obtain a normal basis of $\mathcal S$ as the following
\begin{equation}\label{eq_L1_0}
	\bar{\bf A} = \left[\bar{\bf a}_1,\bar{\bf a}_2,\dots,\bar{\bf a}_d\right] = \left[\frac{{\bf a}_1}{\|{\bf a}_1\|},\frac{{\bf a}_2}{\|{\bf a}_2\|},\dots,\frac{{\bf a}_d}{\|{\bf a}_d\|}\right].
\end{equation}
Furthermore, we can apply Gram–Schmidt process \cite{bjorck1994numerics, hoffmann1989iterative} on $\bar{\bf A}$, or directly on $\bf A$, to obtain an orthonormal basis of $\mathcal S$ as the following
\begin{equation}\label{eq_L1_1}
{\bf v}_i = \frac{\tilde{\bf v}_i}{\|\tilde{\bf v}_i\|}, \quad i = 1,2,\dots,d,
\end{equation}
where
\begin{equation}\label{eq_L1_2}
\tilde{\bf v}_i =\bar{\bf a}_i-\sum_{m=1}^{i-1} \left(\bar{\bf a}_i^{\rm T}{\bf v}_m\right){\bf v}_m, \quad i = 1,2,\dots,d.
\end{equation}
Notice that the index $i$ in \eqref{eq_L1_2} should start from $1$
and increase to $d$,
and if ${\bf a}_i$ is used instead of $\bar{\bf a}_i$ in \eqref{eq_L1_2},
then the result remains the same.
We denote the matrix $\left[{\bf v}_1,{\bf v}_2,\dots,{\bf v}_d\right]$ as $\bf V$.
A natural question is how to measure the similarity between $\bf A$ (or $\bar{\bf A}$) and $\bf V$
as base matrices of the same subspace.


Consider the case where subspace $\mathcal S$ is described by certain data points on it.
In other words, what we have is a set of linearly independent points $\{{\bf a}_i\}_i$
on a latent subspace, rather than an orthonormal basis of it.
In order to calculate the energy of the projection of a new data point $\bf x$ on $\mathcal S$,
we need to first apply the Gram-Schmidt process on $\bf A$ to obtain $\bf V$,
then the energy is $\|{\bf V}^{\rm T}{\bf x}\|^2$.
In cases where the amount of data is huge, or the data are acquired and stored in a distributed way,
the cost of the Gram-Schmidt process is high.
An intuitive way of approximating ${\bf V}$ is to normalize $\bf A$ as shown in \eqref{eq_L1_0},
and then to use the obtained $\bar{\bf A}$ to calculate an approximated projection 
$\bar{\bf A}^{\rm T}{\bf x}$ and its energy $\|\bar{\bf A}^{\rm T}{\bf x}\|^2$.
In such a way, how accurate can the approximation be? 
How to evaluate such approximation?
Furthermore, if we directly use $\|{\bf A}^{\rm T}{\bf x}\|^2$ as an approximation of the energy of the projection,
then how large can the error be?
The answers must depend on some properties of $\bf A$ or $\bar{\bf A}$,
and this work will try to find out such answers.


Such problems are fundamental in cases where random matrices are applied \cite{tulino2004random, petz2004asymptotics}.
According to the conclusions of this work,
if ${\bf A}$ is a random matrix,
we do not have to apply the Gram-Schmidt process to ${\bf A}$,
instead the normalized matrix $\bar{\bf A}$ can be a rather accurate approximation.
For a high dimensional random matrix, even normalization is not needed,
and the matrix itself is able to be a good approximation.


\section{Approximation of orthonormal basis by normal basis}
\label{sec2}


We first study to use the normalized matrix $\bar{\bf A}$ to approximate the
orthonormalized matrix $\bf V$.
The similarity between $\bar{\bf A}$ and an orthonormal matrix is measured by 
$\bar{\bf R} = \bar{\bf A}^{\rm T}\bar{\bf A} - {\bf I}$.
Based on a defined decomposition ${\bf V}-\bar{\bf A}=\bar{\bf A}\bar{\bf U}$,
we use $\bar{\bf U}$ to evaluate the similarity between $\bar{\bf A}$ and $\bf V$.
The following lemma describes the performance of $\bar{\bf U}$
as $\bar{\bf R}\rightarrow 0$.


\begin{lemma}\label{L1}
Let ${\bf V}=\left[{\bf v}_1, {\bf v}_2, \dots, {\bf v}_d\right]$ denote
the Gram-Schmidt orthogonalization
of a column-normalized matrix $\bar{\bf A}=\left[\bar{\bf a}_1, \bar{\bf a}_2, \dots, \bar{\bf a}_d\right]$, 
where $\|\bar{\bf a}_i\|=1, \forall i$. 
Denote $\bar{\bf R} = (\bar r_{ji}) = \bar{\bf A}^{\rm T}\bar{\bf A} - {\bf I}$. 
Then when $\bar{r}_{ji} = \bar{\bf a}_j^{\rm T}\bar{\bf a}_i$ is small enough for $j\ne i$, we can use $\bar{\bf A}$ to approximate ${\bf V}$ with error ${\bf V}-\bar{\bf A}=\bar{\bf A}\bar{\bf U}$, where $\bar{\bf U}=[\bar u_{ji}]\in\mathbb{R}^{d\times d}$ is an upper triangular matrix satisfying
\begin{equation}\label{eq_L1_0_1}
\bar u_{ii} = \bar g_{ii}(\bar{\bf R})\|\bar{\bf R}\|_F^2,\quad \forall i,
\end{equation}
where $\bar g_{ii}(\bar{\bf R}) > 0$ and $\lim_{\bar{\bf R} \to {\bf 0}} \bar g_{ii}(\bar{\bf R}) \le {1}/{4}$, and
\begin{equation}\label{eq_L1_0_2}
\bar u_{ji} = - \bar r_{ji} + \bar g_{ji}(\bar{\bf R})\|\bar{\bf R}\|_F, \quad \forall j<i,
\end{equation}
where $\lim_{\bar{\bf R} \to {\bf 0}} \bar g_{ji}(\bar {\bf R}) = 0$.
\end{lemma}
\begin{proof}
We define $\bf V$ following the Gram-Schmidt process of \eqref{eq_L1_1} and \eqref{eq_L1_2}.
We have ${\bf V}=\bar{\bf A}\bar{\bf G}$, where $\bar{\bf G}$ is an upper triangular matrix. Accordingly, $\bar{\bf U}=\bar{\bf G}-{\bf I}$ is also upper triangular and 
\begin{equation}\label{eq_L1_3}
{\bf v}_i = \bar{\bf a}_i+\sum_{j=1}^i \bar u_{ji}\bar{\bf a}_j.
\end{equation}
Using \eqref{eq_L1_3} and \eqref{eq_L1_2} in \eqref{eq_L1_1}, we have
\begin{equation}\label{eq_L1_4}
        {\bf v}_i = \frac{1}{\|\tilde{\bf v}_i\|} \left(\bar{\bf a}_i-\sum_{m=1}^{i-1} \bar{\bf a}_i^{\rm T}{\bf v}_m\left(\bar{\bf a}_m+\sum_{j=1}^m \bar u_{jm}\bar{\bf a}_j\right)\right).
\end{equation}
By switching the order of the summations, \eqref{eq_L1_4} can be reformulated as
\begin{align}
        {\bf v}_i &= \frac{1}{\|\tilde{\bf v}_i\|} \left(\bar{\bf a}_i 
        - \sum_{j=1}^{i-1}\left(\bar{\bf a}_i^{\rm T}{\bf v}_j+\sum_{m=j}^{i-1}\left(\bar{\bf a}_i^{\rm T}{\bf v}_m\right)\bar u_{jm}\right)
        \bar{\bf a}_j\right)\nonumber\\
        &= \frac{\bar{\bf a}_i}{\|\tilde{\bf v}_i\|} 
        - \sum_{j=1}^{i-1}\frac{\bar{\bf a}_i^{\rm T}{\bf v}_j+\sum_{m=j}^{i-1}\left(\bar{\bf a}_i^{\rm T}{\bf v}_m\right)\bar u_{jm}}{\|\tilde{\bf v}_i\|}\bar{\bf a}_j. \label{eq_L1_5}
\end{align}
Comparing \eqref{eq_L1_3} and \eqref{eq_L1_5}, we readily get
\begin{align}
	\bar u_{ii} &= \frac{1}{\|\tilde{\bf v}_i\|} - 1,\quad \forall i, \label{eq_L1_6}\\
	\bar u_{ji} &= - \frac{1}{\|\tilde{\bf v}_i\|}\left(\bar{\bf a}_i^{\rm T}{\bf v}_j+\sum_{m=j}^{i-1}\left(\bar{\bf a}_i^{\rm T}{\bf v}_m\right)\bar u_{jm}\right), \quad \forall j < i. \label{eq_L1_7}
\end{align}
We will first study \eqref{eq_L1_6} and then turn to \eqref{eq_L1_7}. 
Using \eqref{eq_L1_2} in \eqref{eq_L1_6} and noticing that both $\bar{\bf a}_i$ and ${\bf v}_m$ have been normalized, we have 
\begin{equation}\label{eq_L1_9}
	\bar u_{ii} = \frac{1}{\|\bar{\bf a}_i-\sum_{m=1}^{i-1} \left(\bar{\bf a}_i^{\rm T}{\bf v}_m\right){\bf v}_m\|}-1
	= \frac{1}{\sqrt{1 - \sum_{m=1}^{i-1} \left(\bar{\bf a}_i^{\rm T}{\bf v}_m\right)^2}} - 1. 
\end{equation}
According to the Taylor's series with Peano form of the remainder, i.e.,
$$
f(x) = \frac{1}{\sqrt{1-x}} = 1 + \frac{x}{2} + h(x)x,
$$
where $\lim_{x \to 0} h(x) = 0$, \eqref{eq_L1_9} is approximated by
\begin{equation}
	\bar u_{ii} = \left(\frac12+h(\cdot)\right) \sum_{m=1}^{i-1} \left(\bar{\bf a}_i^{\rm T}{\bf v}_m\right)^2, \label{eq_L1_10}
\end{equation}
where $h\left(\sum_{m=1}^{i-1} \left(\bar{\bf a}_i^{\rm T}{\bf v}_m\right)^2\right)$ is denoted by $h(\cdot)$ for short.
Following \eqref{eq_L1_3} and using the definition of $\bar{\bf R}$, for $m<i$ we have
\begin{equation}\label{eq_L1_11}
        \bar{\bf a}_i^{\rm T}{\bf v}_m = \bar{\bf a}_i^{\rm T}\bar{\bf a}_m + \sum_{k=1}^m \bar u_{km}\bar{\bf a}_i^{\rm T}\bar{\bf a}_k 
        = \bar r_{mi} + \sum_{k=1}^m \bar u_{km}\bar r_{ki}. 
\end{equation}
Using \eqref{eq_L1_11} in \eqref{eq_L1_10}, we have
\begin{align}
	\bar u_{ii} &= \left(\frac12+h(\cdot)\right) \sum_{m=1}^{i-1} \left(\bar r_{mi} + \sum_{k=1}^m \bar u_{km}\bar r_{ki}\right)^2\nonumber\\
	&=  \left(\frac12+h(\cdot)\right) \left(\sum_{m=1}^{i-1}\bar r_{mi}^2 + \sum_{m=1}^{i-1}\left(\left(\sum_{k=1}^m \bar u_{km}\bar r_{ki}\right)^2 + 2\sum_{k=1}^m \bar u_{km}\bar r_{mi}\bar r_{ki}\right)\right). \label{eq_L1_12}
\end{align}
Because of the symmetry of $\bar{\bf R}$, the first summation in the RHS of \eqref{eq_L1_12} is bounded by $\frac12\|\bar{\bf R}\|_F^2$. Furthermore, the second summation, which is composed of squares and products of $\bar r_{pq}$, must be bounded by $\epsilon_1\|\bar{\bf R}\|_F^2$, where $\epsilon_1$ is a small quantity. 
Consequently, we have
\begin{equation}
	\bar u_{ii} = \bar g_{ii}(\bar{\bf R})\|\bar{\bf R}\|_F^2 \le  \left(\frac12+h(\cdot)\right) \left(\frac12 + \epsilon_1\right)\|\bar{\bf R}\|_F^2, \label{eq_L1_13}
\end{equation}
where
\begin{equation}
	\lim_{\bar{\bf R}\rightarrow {\bf 0}}\bar g_{ii}(\bar{\bf R}) \le \frac14, \label{eq_L1_14}
\end{equation}
because $h(\cdot)$ tends to $0$ as $\bar{\bf R}$ approaches $\bf 0$. We then complete the first part of the lemma.

Next we will study \eqref{eq_L1_7}. Using \eqref{eq_L1_6} and \eqref{eq_L1_11} in \eqref{eq_L1_7}, we have
\begin{align}
	\bar u_{ji} &= -(1+\bar u_{ii})\left(\bar r_{ji} + \sum_{k=1}^j \bar u_{kj}\bar r_{ki}+\sum_{m=j}^{i-1}\left(\bar r_{mi} + \sum_{l=1}^m \bar u_{lm}\bar r_{li}\right)\bar u_{jm}\right)\nonumber\\
	&= -(1+\bar u_{ii})\left(\bar r_{ji} + \sum_{k=1}^j \bar u_{kj}\bar r_{ki}+\sum_{m=j}^{i-1}\bar u_{jm}\bar r_{mi}+\sum_{m=j}^{i-1}\sum_{l=1}^m \bar u_{lm}\bar u_{jm}\bar r_{li}\right), \quad\forall j<i.\label{eq_L1_15}
\end{align}
Notice that the summations in \eqref{eq_L1_15}, which are composed of $\bar r_{pq}$, must be bounded by $\epsilon_2\|\bar{\bf R}\|_F$, where $\epsilon_2$ is a small quantity. 
Plugging \eqref{eq_L1_13} and \eqref{eq_L1_14} into \eqref{eq_L1_15}, we have
\begin{equation}\label{eq_L1_16}
	\bar u_{ji} = -\left(1+\bar g_{ii}(\bar{\bf R}\right)\|\bar{\bf R}\|_F^2)\left(\bar r_{ji} + \epsilon_2\|\bar{\bf R}\|_F\right)
	= -\bar r_{ji} + \bar g_{ji}(\bar{\bf R})\|\bar{\bf R}\|_F,
\end{equation}
where
\begin{equation}
	\lim_{\bar{\bf R}\rightarrow{\bf 0}}\bar g_{ji}(\bar{\bf R}) = 0.
\end{equation}
The second part of the lemma is proved.
\end{proof}

Lemma \ref{L1} unveils that, when $\bar{\bf A}$ approaches an orthonormal basis,
i.e., $\bar{\bf R}$ approaches $\bf 0$, the diagonal elements of $\bar{\bf U}$ go to zero,
and they are of the same order as $\|\bar{\bf R}\|_F^2$.
At the same time, the off-diagonal elements go to $-\bar{\bf R}$,
and the differences are of a higher order than $\|\bar{\bf R}\|_F$.


\begin{remark}
Notice that the error that we define is based on $\bar{\bf A}$ rather than $\bf V$,
i.e., we define ${\bf V}-\bar{\bf A}=\bar{\bf A}\bar{\bf U}$ rather than 
${\bf V}-\bar{\bf A}={\bf V}\bar{\bf U}$.
The reason is that $\bar{\bf A}$ is at hand and can be easily obtained,
while $\bf V$ is expensive to calculate.
It would contradict our purpose of reducing the computation complexity,
if $\bf V$ were used here.

\end{remark}


For the energy of the projection of a vector onto a subspace,
based on Lemma \ref{L1}, we can obtain the approximation error of using
$\bar{\bf A}$ instead of $\bf V$.
The following corollary gives an upper bound on such error.


\begin{corollary}\label{C1}
Following the definition of Lemma \ref{L1}, we use a column-normalized matrix $\bar{\bf A}=\left[\bar{\bf a}_1, \bar{\bf a}_2, \dots, \bar{\bf a}_d\right]$ to approximate its orthonormal matrix ${\bf V}=\left[{\bf v}_1, {\bf v}_2, \dots, {\bf v}_d\right]$ gotten through Gram-Schmidt process. 
The approximation error ${\bf V}-\bar{\bf A}=\bar{\bf A}\bar{\bf U}$. 
For an arbitrary vector ${\bf x} \in\mathbb{R}^n$, we conclude that
\begin{equation}\label{eq_C1_1}
\left|\|{\bf V}^{\rm T}{\bf x}\|^2-\|\bar{\bf A}^{\rm T}{\bf x}\|^2\right| \le d\|\bar{\bf A}^{\rm T}{\bf x}\|^2\max{\bar{\bf R}} + \epsilon(\bar{\bf R})\|\bar{\bf R}\|_F,
\end{equation}
where $\lim_{\bar{{\bf R}} \to {\bf 0}} \epsilon(\bar{\bf R}) = 0$.
\end{corollary}
\begin{proof}
Let ${\bf b}_0={\bf V}^{\rm T}{\bf x}$ and ${\bf b}=\bar{\bf A}^{\rm T}{\bf x}$. We have
\begin{equation}\label{eq_C1_2}
        \textrm{LHS\ of\ \eqref{eq_C1_1}}
        = \left|\|{\bf b}_0\|^2-\|{\bf b}\|^2\right|
        = \left|\sum_{i=1}^d (b_{0,i}^2 - b_i^2)\right|
        \le \sum_{i=1}^d \left|b_{0,i}^2 - b_i^2\right|.
\end{equation}
By further defining ${\bf c}={\bf b}_0-{\bf b}=\bar{\bf U}^{\rm T}{\bf b}$, where $c_i=\sum_{m=1}^i \bar u_{mi}b_m$, we have
\begin{align}
        \textrm{RHS\ of\ \eqref{eq_C1_2}} =& \sum_{i=1}^d \left|c_i(2b_i+c_i)\right| \nonumber\\
        \le& \sum_{i=1}^d \sum_{m=1}^i |\bar u_{mi}|(b_i^2+b_m^2) + \sum_{i=1}^d c_i^2 \nonumber\\
        \le& \sum_{i=1}^d \sum_{m=1}^{i-1} |\bar u_{mi}|(b_i^2+b_m^2) + 2\sum_{i=1}^d |\bar u_{ii}|b_i^2 + \sum_{i=1}^d i\sum_{m=1}^i \bar u_{mi}^2b_m^2.\label{eq_C1_3}
\end{align}
Let's first check the third item in the RHS of \eqref{eq_C1_3}. 
By using \eqref{eq_L1_0_1} and \eqref{eq_L1_0_2}, we have
\begin{align}
	\sum_{i=1}^d i\sum_{m=1}^i \bar u_{mi}^2b_m^2 &= \sum_{i=1}^di\sum_{m=1}^{i-1}\bar u_{mi}^2b_m^2 + \sum_{i=1}^di\bar u_{ii}^2b_i^2\nonumber\\
	&= \sum_{i=1}^di\sum_{m=1}^{i-1}\left(-\bar r_{mi}+\bar g_{mi}(\bar{\bf R})\|\bar{\bf R}\|_F\right)^2b_m^2 + \sum_{i=1}^di\bar g_{ii}^2(\bar{\bf R})\|\bar{\bf R}\|_F^4b_i^2\label{eq_C1_4}\\
	&= \epsilon'(\bar{\bf R})\|\bar{\bf R}\|_F,\label{eq_C1_5}
\end{align}
where $\lim_{\bar{\bf R} \to {\bf 0}} \epsilon'(\bar{\bf R}) = 0$. 
Equation \eqref{eq_C1_5} is derived because all items in \eqref{eq_C1_4} are of higher order of $\|\bar{\bf R}\|_F$. 
Following the similar way, we adopt \eqref{eq_L1_0_1}, \eqref{eq_L1_0_2}, and \eqref{eq_C1_5} in \eqref{eq_C1_3},
\begin{align}
        \textrm{RHS\ of\ \eqref{eq_C1_3}} \le& \sum_{i=1}^d \sum_{m=1}^{i-1} (|\bar r_{mi}| + |\bar g_{mi}(\bar{\bf R})|\|\bar{\bf R}\|_F) (b_i^2+b_m^2) + 2\sum_{i=1}^d |\bar g_{ii}(\bar{\bf R})|\|\bar{\bf R}\|_F^2b_i^2 + \epsilon'(\bar{\bf R})\|\bar{\bf R}\|_F \nonumber\\
        \le& \left(\sum_{i=1}^d \sum_{m=1}^{i-1} b_i^2+b_m^2\right)\max(\bar{\bf R}) \nonumber\\
        &  + \sum_{i=1}^d \sum_{m=1}^{i-1} |\bar g_{mi}(\bar{\bf R})|\|\bar{\bf R}\|_F(b_i^2+b_m^2) + 2\sum_{i=1}^d |\bar g_{ii}(\bar{\bf R})|\|\bar{\bf R}\|_F^2b_i^2 + \epsilon'(\bar{\bf R})\|\bar{\bf R}\|_F \label{eq_C1_6}\\
        \le& (d-1)\|{\bf b}\|^2\max{\bar{\bf R}} + \epsilon(\bar{\bf R})\|\bar{\bf R}\|_F \label{eq_C1_7} \le \textrm{RHS\ of\ \eqref{eq_C1_1}},
\end{align}
where $\lim_{{\bf R} \to {\bf 0}} \epsilon(\bar{\bf R}) = 0$. 
Equation \eqref{eq_C1_7} is derived because the last three items in \eqref{eq_C1_6} are of higher order of $\|\bar{\bf R}\|_F$. 
\end{proof}

\begin{remark}
${\bf V}^{\rm T}{\bf x}$ denotes the projection of a vector ${\bf x}$ in the subspace of ${\bf V}$. 
Then Corollary \ref{C1} shows that the relative error for using $\bar{\bf A}^{\rm T}{\bf x}$ to estimate the projected energy is $d\max{\bar{\bf R}}$.
\end{remark}

\begin{example}
Given a random matrix ${\bm \Phi} \in \mathbb{R}^{n \times k}$, whose entries are independent standard normal random variables, we can estimate the truncated Haar matrix \cite{tulino2004random, petz2004asymptotics} by $\bar{\bm \Phi}$ through normalizing the columns of ${\bm \Phi}$. 
According to Lemma \ref{L2}, we can easily find that, with probability at least $1 - ({k(k-1)}/{2})\exp\left({-{n\varepsilon^2}/{2}}\right)$, the inner product of any two columns of $\bar{\bm \Phi}$ is less than $\varepsilon$. 
Then, according to Lemma \ref{L1},  the Frobenius norm of the estimating error is less than $\sqrt{{k(k-1)}/{2}}\varepsilon + o(\varepsilon)$. 
On the other hand, according to Corollary \ref{C1}, we can consider $\bar{\bm \Phi}\bar{\bm \Phi}^{\rm T}$ as the projection matrix of ${\bm \Phi}$.
\end{example}

\begin{lemma}\label{L2} 
Since the normalized Gaussian random vector is uniformly distributed on the sphere, according to the concentration of measure on the sphere \cite{levy1951problemes, schmidt1948brunn}, we have $\mathbb{P}\left\{|\cos\theta|>\varepsilon\right\}\le \exp\left({-{n\varepsilon^2}/{2}}\right)$, where $\theta$ denotes the angle between two independent Gaussian random vectors, whose elements are independent standard normal random variables.
\end{lemma}

\section{Approximation of orthonormal basis by arbitrary basis}
\label{sec3}

Based on the previous section, we discuss the error of using the original matrix $\bf A$
as an approximation of its orthonormalization $\bf V$.
To begin with, we define $({\bf R}, {\bf W})$ to measure the similarity between 
${\bf A}$ and $\bf V$.
Then the matrix ${\bf U}$ is used to describe the similarity between 
${\bf A}$ and $\bf V$, where ${\bf V}-{\bf A}={\bf A}{\bf U}$.
The following corollary describes the performance of ${\bf U}$
as ${\bf R}\rightarrow 0$ and ${\bf W}\rightarrow{\bf I}$.


\begin{corollary}\label{C2}
Let ${\bf V}=\left[{\bf v}_1, {\bf v}_2, \dots, {\bf v}_d\right]$ denote the orthonormal matrix of an arbitrary matrix ${\bf A}=\left[{\bf a}_1, {\bf a}_2, \dots, {\bf a}_d\right]$ gotten through the Gram-Schmidt process. 
Let ${\bf W}$ be a diagonal matrix with $w_{ii} = \|{\bf a}_i\|^2$, and ${\bf R} = (r_{ij}) = {\bf A}^{\rm T}{\bf A} - {\bf W}$. 
When $\bf W$ approaches to $\bf I$, and $r_{ji} = {\bf a}_j^{\rm T}{\bf a}_i$ is small enough for $j \ne i$, we can use ${\bf A}$ to approximate ${\bf V}$ with error ${\bf V}-{\bf A}={\bf A}{\bf U}$, where ${\bf U}=(u_{ji})\in\mathbb{R}^{d\times d}$ is an upper triangular matrix satisfying
\begin{equation}\label{eq_C2_1}
u_{ii} = \frac{1-w_{ii}}{2} + h({\bf R},{\bf W})(1-w_{ii}) + g_{ii}({\bf R}, {\bf W})\|{\bf R}\|_F^2,\quad \forall i,
\end{equation}
where $\lim_{{\bf R} \to {\bf 0}, {\bf W} \to {\bf I}} h({\bf R}, {\bf W}) = 0$ and $\lim_{{\bf R} \to {\bf 0},{\bf W} \to {\bf I}} g_{ii}({\bf R}, {\bf W})\le 1/4$, and
\begin{equation}\label{eq_C2_2}
u_{ji} = - r_{ji} + g_{ji}({\bf R}, {\bf W})\|{\bf R}\|_F, \quad \forall j< i,
\end{equation}
where $\lim_{{\bf R} \to {\bf 0}, {\bf W} \to {\bf I}} g_{ji}({\bf R}, {\bf W}) = 0$.
\end{corollary}
\begin{proof}
The proof follows a similar routine as that of Lemma \ref{L1}, where the variables with bar in the proof of Lemma \ref{L1} are exactly the counterparts of the variables here. 
Therefore we will only highlight those different. 
Referring to the deduction of \eqref{eq_L1_6} and \eqref{eq_L1_7} in the proof of Lemma \ref{L1}, we have
\begin{align}
	u_{ii} &= \frac{1}{\|\tilde{\bf v}_i\|} - 1, \quad  \forall i, \label{eq_C2_3}\\
	u_{ji} &= - \frac{1}{\|\tilde{\bf v}_i\|}\left({\bf a}_i^{\rm T}{\bf v}_j+\sum_{m=j}^{i-1}\left({\bf a}_i^{\rm T}{\bf v}_m\right) u_{jm}\right), \quad \forall j < i. \label{eq_C2_4}
\end{align}
where
\begin{align}
\tilde{\bf v}_i &={\bf a}_i-\sum_{m=1}^{i-1} \left({\bf a}_i^{\rm T}{\bf v}_m\right){\bf v}_m, \quad \forall i, \label{eq_C2_5}\\
{\bf a}_i^{\rm T}{\bf v}_m &= {\bf a}_i^{\rm T}{\bf a}_m + \sum_{k=1}^m  u_{km}{\bf a}_i^{\rm T}{\bf a}_k 
        =  r_{mi} + \sum_{k=1}^m  u_{km} r_{ki},\quad \forall m < i. \label{eq_C2_51}
\end{align}
We will first check \eqref{eq_C2_3} and then \eqref{eq_C2_4}. 
Noticing that ${\bf a}_i$ is not normalized, we have
 \begin{equation}\label{eq_C2_6}
	u_{ii} = \frac{1}{\|{\bf a}_i-\sum_{m=1}^{i-1} \left({\bf a}_i^{\rm T}{\bf v}_m\right){\bf v}_m\|}-1
	= \frac{1}{\sqrt{w_{ii} - \sum_{m=1}^{i-1} \left({\bf a}_i^{\rm T}{\bf v}_m\right)^2}}-1.
\end{equation}
Using the Taylor's series with Peano form of the remainder in \eqref{eq_C2_6}, we have
\begin{equation}\label{eq_C2_7}
	u_{ii} = \left(\frac12 + h({\bf R},{\bf W})\right)\left(1-w_{ii} + \sum_{m=1}^{i-1} \left({\bf a}_i^{\rm T}{\bf v}_m\right)^2\right),
\end{equation}
where, without confusing, $h\left(1-w_{ii} + \sum_{m=1}^{i-1} \left({\bf a}_i^{\rm T}{\bf v}_m\right)^2\right)$ is denoted as a function of ${\bf R}$ and ${\bf W}$ for better understanding. 
Using \eqref{eq_C2_51} in \eqref{eq_C2_7} and referring to deduction of \eqref{eq_L1_12}, we have
\begin{equation}\label{eq_C2_8}
	u_{ii} = \frac{1-w_{ii}}{2} + h({\bf R},{\bf W})(1-w_{ii}) + g_{ii}({\bf R}, {\bf W})\|{\bf R}\|_F^2,
\end{equation}
where $\lim_{{\bf R} \to {\bf 0}, {\bf W} \to {\bf I}} h({\bf R}, {\bf W}) = 0$ and $\lim_{{\bf R} \to {\bf 0}, {\bf W} \to {\bf I}} g_{ii}({\bf R}, {\bf W}) \le 1/4$. 

Now we will study \eqref{eq_C2_4}. 
Plugging \eqref{eq_C2_3}, \eqref{eq_C2_51}, and \eqref{eq_C2_8} in \eqref{eq_C2_4} and referring to the deduction of \eqref{eq_L1_15} and \eqref{eq_L1_16}, we have
\begin{align}
	u_{ji} &= -(1+u_{ii})\left({\bf a}_i^{\rm T}{\bf v}_j+\sum_{m=j}^{i-1}\left({\bf a}_i^{\rm T}{\bf v}_m\right) u_{jm}\right) \nonumber\\
	&= -\left(1 + \frac{1-w_{ii}}{2} + h({\bf R}, {\bf W})(1-w_{ii}) + g_{ii}({\bf R},{\bf W})\|{\bf R}\|_F^2\right) \nonumber\\
	&\quad\cdot\left( r_{ji} + \sum_{k=1}^j  u_{kj} r_{ki}+\sum_{m=j}^{i-1} u_{jm} r_{mi}+\sum_{m=j}^{i-1}\sum_{l=1}^m  u_{lm} u_{jm} r_{li}\right) \nonumber \\
	&= -r_{ji} + g_{ji}({\bf R}, {\bf W})\|{\bf R}\|_F, \quad \forall j<i,
\end{align}
where $\lim_{{\bf R} \to {\bf 0}, {\bf W} \to {\bf I}} g_{ji}({\bf R}, {\bf W}) = 0$. We then complete the proof.
\end{proof}

\begin{example}
Given a random matrix ${\bm \Phi} \in \mathbb{R}^{n \times k}$, whose entries are independent standard normal random variables, we can also use $({1}/{\sqrt{n}}){\bm \Phi}$ to approximate the truncated Haar matrix \cite{tulino2004random, petz2004asymptotics}. According to Corollary \ref{C2} and Law of Large Number, with high probability, the error is small enough when $n$ is large enough. 
Notice that the random matrix here is different from the measurement matrix in Compressed Sensing (CS) \cite{candes2008restricted, donoho2006compressed, johnson1984extensions}, since here we need $n \gg k$. 
\end{example}

\begin{remark}
If a Gaussian random matrix is orthonormalized, then its columns (and even entries) are not independent anymore.
Therefore, the orthonormal matrix no longer satisfies useful properties of Gaussian matrices.
According to the proposed theoretical analysis, a Gaussian matrix can be an approximation of its
orthonormalization. Certain error is inevitable, but it can be small enough, and the independency between
columns (and even entries) is preserved.

\end{remark}

\begin{example}
As an application, the conclusions of this work can be used to prove the restricted isometric property of
random projection of a finite number of subspaces \cite{2017ICASSP}, where the detailed proofs can be found in \cite{FullVersion}.
\end{example}

\bibliographystyle{IEEEtran}
\bibliography{IEEEabrv,mybibfile}

\vfill\pagebreak

\end{document}